\documentclass{amsart}
\usepackage{amsfonts}

\usepackage{amscd,amssymb,amsmath,graphicx,verbatim}
\usepackage[dvips]{hyperref}
\usepackage[TS1,OT1,T1]{fontenc}

\newtheorem{theorem}{Theorem}[section]

\newtheorem{corollary}[theorem]{Corollary}
\newtheorem{proposition}[theorem]{Proposition}

\theoremstyle{definition}
\newtheorem{definition}[theorem]{Definition}

\theoremstyle{remark}
\newtheorem{remark}[theorem]{Remark}

\numberwithin{equation}{section}

\begin{document}

\title{Chaos for Cowen-Douglas operators}
\author{Bingzhe Hou }
\address{Bingzhe Hou, Department of Mathematics , Jilin university, 130012, Changchun, P.R.China} \email{houbz@jlu.edu.cn}

\author{Puyu Cui}
\address{Puyu Cui, Department of Mathematics , Dalian university of technology, 116024, Dalian, P.R.China} \email{cuipuyu1234@sina.com.cn}

\author {Yang Cao}
\address{Yang Cao, Department of Mathematics , Jilin university, 130012, Changchun, P.R.China} \email{caoyang@jlu.edu.cn}

\date{Jan. 4, 2009}
\thanks{The first author
is supported by the Youth Foundation of Department of Mathematics,
Jilin university.}
\begin{abstract}
In this article, we provide a sufficient condition which gives
Devaney chaos and distributional chaos for Cowen-Douglas operators.
In fact, we obtain a distributionally chaotic criterion for bounded
linear operators on Banach spaces.
\end{abstract}
\maketitle

\section{Introduction and preliminaries}
A discrete dynamical system is simply a continuous mapping $f:
X\rightarrow X$ where $X$ is a complete separable metric space. For
$x\in X$, the orbit of $x$ under $f$ is
$Orb(f,x)=\{x,f(x),f^{2}(x),\ldots\}$ where $f^{n}= f\circ f\circ
\cdots \circ f $ is the $n^{th}$ iterate of $f$ obtained by
composing $f$ with $n$ times.

Recall that $f$ is transitive if for any two non-empty open sets
$U,V$ in $X$, there exists an integer $n\geq1$ such that
$f^{n}(U)\cap V\neq \phi$. It is well known that, in a complete
metric space without isolated points, transitivity is equivalent to
the existence of dense orbit (\cite{Silverman}). $f$ is weakly
mixing if $(f\times f, X\times X$ is transitive. $f$ is strongly
mixing if for any two non-empty open sets $U,V$ in $X$, there exists
an integer $m\geq1$ such that $f^{n}(U)\cap V\neq \phi$ for every
$n\geq m$. $f$ has sensitive dependence on initial conditions (or
simply $f$ is sensitive)if there is a constant $\delta>0$ such that
for any $x\in X$ and any neighborhood $U$ of $x$, there exists a
point $y\in X$ such that $d(f^{n}(x),f^{n}(y))> \delta$, where $d$
denotes the metric on $X$.

In 1975, Li and Yorke \cite{L-Y} observed complicated dynamical
behavior for the class of interval maps with period 3. This
phenomena is currently known under the name of Li-Yorke chaos.
Therefrom, several kinds of chaos were well studied. In the present
article, we focus on Devaney chaos and distributional chaos.

Following Devaney \cite{Devaney},

\begin{definition}

Let $(X,f)$ be a dynamical system. $f$ is chaotic if

$(D1)$ $f$ is transitive;

$(D2)$ the periodic points for $f$ are dense in $X$; and

$(D3)$ $f$ has sensitive dependence on initial conditions.

\end{definition}

It was shown by Banks et. al. (\cite{Banks}) that $(D1) + (D2)$
implies $(D3)$ for any aperiodic system. Devaney chaos is a stronger
version chaos than Li-Yorke chaos, which is given by Huang and Ye
\cite{Huang} and  Mai \cite{Mai}.

From Schweizer and Sm\'{\i}tal's paper \cite{S-S}, distributional
chaos is defined in the following way.

For any pair $\{x,y\}\subset X$ and any $n\in \mathbb{N}$, define
distributional function $F^{n}_{xy}:\mathbb{R}\rightarrow [0,1]$:
\begin{equation*}
F^{n}_{xy}(\tau)=\frac{1}{n}\#\{0\leq i\leq n:
d(f^{i}(x),f^{i}(y))<\tau\}.
\end{equation*}
Furthermore, define
\begin{align*}
F_{xy}(\tau)=\liminf\limits_{n\rightarrow\infty}F^{n}_{xy}(\tau), \\
F_{xy}^{*}(\tau)=\limsup\limits_{n\rightarrow\infty}F^{n}_{xy}(\tau)
\end{align*}
Both $F_{xy}$ and $F_{xy}^{*}$ are nondecreasing functions and may
be viewed as cumulative probability distributional functions
satisfying $F_{xy}(\tau)=F_{xy}^{*}(\tau)=0$ for $\tau<0$.

\begin{definition}
$\{x,y\}\subset X$ is said to be a distributionally chaotic pair, if
$$
F_{xy}^{*}(\tau)\equiv 1, \ \ \forall \ \ \tau>0 \quad and \quad
F_{xy}(\epsilon)=0, \ \ \exists \ \ \epsilon>0.
$$
Furthermore, $f$ is called distributionally chaotic, if there exists
an uncountable subset $D\subseteq X$ such that each pair of two
distinct points is a distributionally chaotic pair. Moreover, $D$ is
called a distributionally  $\epsilon$-scrambled set.
\end{definition}

Distributional chaos always implies Li-Yorke chaos, as it requires
more complicated statistical dependence between orbits than the
existence of points which are proximal but not asymptotic. The
converse implication is not true in general. However in practice,
even in the simple case of Li-Yorke chaos, it might be quite
difficult to prove chaotic behavior from the very definition. Such
attempts have been made in the context of linear operators (see
\cite{Duan, Fu}). Further results of \cite{Duan} were extended in
\cite{Opr} to distributional chaos for the annihilation operator of
a quantum harmonic oscillator. More about distributional chaos, one
can see \cite{Smital1, Smital2, Liao1, Liao2, Wang}.

Let $\mathcal {H}$ be a complex separable Hilbert space and
$\mathcal {L}(\mathcal {H})$ denote the collection of bounded linear
operators on $\mathcal {H}$. We are interesting in a family of
operators given by Cowen and Douglas \cite{Cowen}.

\begin{definition}
For $\Omega$ a connected open subset of $\mathbb{C}$ and $n$ a
positive integer, let $\mathcal {B}_{n}(\Omega)$ denotes the
operators $T$ in $\mathcal {L}(\mathcal {H})$ which satisfy:

(a) $\Omega \subseteq \sigma(T)=\{\omega \in \mathbb{C}: T-\omega
 {\ not \ invertible}\}$;

(b) $ran(T-\omega)=\mathcal {H} \ for \ \omega \ in \ \Omega$;

(c) $\bigvee ker_{\omega\in \Omega}(T-\omega)=\mathcal {H}$; and

(d) $dim \ ker(T-\omega)=n$ for $\omega$ in $\Omega$.
\end{definition}

We have known some properties of these operators from \cite{Cowen}.
\begin{proposition}\label{one span}
Let $T\in \mathcal {B}_{n}(\Omega)$ and $\omega_{0}\in \Omega$. Then
 $\bigvee\limits_{k=1}^{\infty} ker(T-\omega_{0})^{k}=\mathcal {H}$.
\end{proposition}

\begin{proposition}\label{local span}
If $\Omega_{0} \subseteq \Omega$ is a bounded connected open subset
of $\mathbb{C}$, then $B_{n}(\Omega) \subseteq B_{n}(\Omega_{0})$.
\end{proposition}

In this paper, $S$ is always used to denote the unit circle in
$\mathbb{C}$. In the next section, we provide a sufficient condition
$\Omega\cap S\neq \phi$ which gives Devaney chaos for Cowen-Douglas
operators. In the last section, we obtain a distributionally chaotic
criterion for bounded linear operators on Banach spaces. Applied by
this distributionally chaotic criterion, $\Omega\cap S\neq \phi$ is
also a sufficient condition which gives distributional chaos for
Cowen-Douglas operators.

\section{Devaney chaos for Cowen-Douglas operators}

\begin{proposition}\label{str mixing}
Let $T\in \mathcal {B}_{n}(\Omega)$. If $\Omega\cap S\neq \phi$,
then $T$ is strongly mixing.
\end{proposition}

\begin{proof}
Suppose $U$ and $V$ be arbitrary open subsets in $\mathcal {H}$. We
have $\epsilon>0$ and open subsets $U'$ and $V'$ such that
$$B(u',\epsilon)\subseteq U  \ \ and \ \ B(v',\epsilon)\subseteq V ,$$
for any $u'\in U'$ and any $v'\in V'$. Since $\Omega$ is a connected
open subset and $\Omega\cap S\neq \phi$, there are two bounded
connected open subsets $\Omega_{1}$ and $\Omega_{2}$ in $\Omega$
such that
$$\sup\limits_{\alpha\in\Omega_{1}}|\alpha|=\lambda<1 \ \ and \ \ \inf\limits_{\beta\in\Omega_{2}}|\beta|=\rho>1.$$
By proposition \ref{local span}, there exist two points $x\in U'$
and $y\in V'$ with the following forms:
$$x=\sum\limits_{i=1}^{t}x_{i} \ \ and \ \ y=\sum\limits_{j=1}^{l}y_{j},$$
where $x_{i}\in  {ker(T-\lambda_{i})}, \ {\lambda_{i}\in
\Omega_{1}}\ and \ \ y_{j}\in  {ker(T-\rho_{j})}, \ \rho_{j}\in
\Omega_{2} $.

Now let
$M=max\{\sum\limits_{i=1}^{t}\|x_{i}\|,\sum\limits_{j=1}^{l}\|y_{j}\|
\}$. Then there is a positive integer $N$ such that for each $k\geq
N$,
$$\lambda^{k}<\epsilon/M\ \ \ \  and \ \ \ \ \rho^{-k}<\epsilon/M.$$
Given any $k\geq N$, let
$u(k)=x+\sum\limits_{j=1}^{l}\rho_{j}^{-k}y_{j}$. Obviously,
$$\|u(k)-x\|=\|\sum\limits_{j=1}^{l}\rho_{j}^{-k}y_{j}\|\leq\sum\limits_{j=1}^{l}|\rho_{j}^{-k}|\cdot\|y_{j}\|
\leq\rho^{-k}\sum\limits_{j=1}^{l}\|y_{j}\|<\epsilon,$$ so $u(k)\in
U$. On the other hand,
$$\|T^{k}u(k)-y\|=\|\sum\limits_{i=1}^{t}\lambda_{i}^{k}x_{i}\|\leq\sum\limits_{i=1}^{t}|\lambda_{i}^{k}|\cdot\|x_{i}\|
\leq\lambda^{k}\sum\limits_{i=1}^{t}\|x_{i}\|<\epsilon,$$ that
implies $T^{k}u(k) \in V$. Hence $T^{k}(U)\cap V\neq\phi$ and
consequently $T$ is strongly mixing.

\end{proof}

\begin{proposition}\label{dense per}
Let $T\in \mathcal {B}_{n}(\Omega)$. If $\Omega\cap S\neq \phi$,
then $Per(T)$ is dense in $\mathcal {H}$.
\end{proposition}

\begin{proof}
Let $\Delta=\{e^{2\pi ri}: for \ all \ rational \ numbers \ \ r\}$.
Then $\Delta$ is  dense in $S$, and one can see that for each
$\delta\in \Delta$, there exists a positive integer $m(\delta)$ such
that $\delta$ is a root of the equation $z^{m(\delta)}=1$. Since
$\Omega$ is a connected open subset and $\Omega\cap S\neq \phi$, we
have $\Omega\cap \Delta \neq \phi$. Now let $s\in\Omega\cap \Delta$.
If $x\in ker(T-s)^{k}$ for any $k$, then $T^{km(s)}(x)=x$ and hence
$\bigcup\limits_{k=1}^{\infty} ker(T-s)^{k}\subseteq Per(T)$.
Therefore, by proposition \ref{one span}, $Per(T)$ is dense in
$\mathcal {H}$.
\end{proof}
By Proposition \ref{str mixing} and \ref{dense per}, one can see the
following result immediately.
\begin{theorem}\label{De-chaos}
Let $T\in \mathcal {B}_{n}(\Omega)$. If $\Omega\cap S\neq \phi$,
then $T$ is Devaney chaotic.
\end{theorem}

\begin{remark}
Notice that $\Omega\cap S\neq \phi$ is not a necessary condition for
Devaney chaos for Cowen-Douglas operators. As well-known, the
backward shift operator $T$, with the weight sequence
$\{\omega_{n}={\frac{n+1}{n}}\}_{n=1}^{\infty}$, is a Devaney
chaotic Cowen-Douglas operator. However, the largest connected open
domain $\Omega$ for $T$, which admits $T\in B_1(\Omega)$, is the
unit open disk and hence is disjoint with $S$.
\end{remark}

\section{Distributionally Chaotic Criterion and its application on  Cowen-Douglas operators}

First of all, we'll give a new concept which is very useful to prove
an bounded linear operator is distributional chaotic.

\begin{definition}
Let $X$ be a Banach space and let $T\in \mathcal {L}(X)$. $T$ is
called norm-unimodal, if we have a constant $\gamma
>1$ such that for any $m\in\mathbb{N}$, there exists $x_m\in
X$ satisfying

$(NU1) \ \ \  \ \ \lim\limits_{k\rightarrow\infty}\|T^kx_m\|=0, $

$(NU2)\ \ \  \ \ \ \| T^ix_m \|\geq \gamma^i\|x_m\|, \ \
i=1,2,\ldots,m. $

Furthermore, such $\gamma$ is said to be a norm-unimodal constant
for the norm-unimodal operator $T$.
\end{definition}

\begin{remark}
If $x$ is the point referred to in the above definition, then for
any $c\in \mathbb{C}$, $cx$ has the same properties as $x$ because
of the linearity of $T$. Therefore, we can select a point with
arbitrary non-zero norm satisfying the same conditions.
\end{remark}

\begin{theorem}[Distributionally Chaotic Criterion]\label{D-C-C}
Let $X$ be a Banach space and let $T\in \mathcal {L}(X)$. If $T$ is
norm-unimodal, then $T$ is distributionally chaotic.
\end{theorem}

\begin{proof}

Let $R=\|T\|$ and let $\gamma$  be a norm-unimodal constant for $T$.
Suppose $\{\epsilon_k\}_{k=1}^{\infty}$ be a sequence of positive
numbers decreasing to zero. First of all, fix $N_1\in\mathbb{N}$
(for example, set $N_1=2$). Then there is $x_1$ such that
$\|x_1\|=1$ and
$$
 \lim\limits_{k\rightarrow\infty}\|T^kx_1\|=0, \ \ and \ \ \ \|
T^ix_1 \|\geq \gamma^i\|x_1\|, \ \ i=1,\ldots,N_1.
$$
So we can choose $M_1$ such that $\|T^nx\|<\epsilon_1$ for any
$n\geq M_1$. For convenience, let $N'_1=0$. Then $\| T^ix_1 \|\geq
1, \ \ i=N'_1,\ldots,N_1.$

Now we'll construct a sequence of points $\{x_k\}_{k=1}^{\infty}$
associated with three sequences of integers
$\{N_k\}_{k=1}^{\infty}$, $\{N'_k\}_{k=1}^{\infty}$and
$\{M_k\}_{k=1}^{\infty}$ such that for every $k\geq2$

(I) $\|x_k\|=R^{-M_{k-1}}\cdot2^{-k}\cdot\epsilon_{k-1}$ ;

(II) $\| T^ix_k \|\geq \gamma^i\|x_k\|, \ \ i=1,\ldots,N_k $;

(III) $\gamma^{N'_k}\cdot
R^{-M_{k-1}}\cdot2^{-k}\cdot\epsilon_{k-1}>1$;

(IV) $\frac{N_k-N'_k}{N_k}>\frac{k-1}{k}$;

(V) $\sum\limits_{j=1}^k\|T^nx_j\|<\epsilon_k$, for any $n\geq M_k$.
\par
\par
Select $N'_2\in\mathbb{N}$ with $\gamma^{N'_2}\cdot
R^{-M_1}\cdot2^{-2}\cdot\epsilon_1>1$. Consequently, we have
$N_2\in\mathbb{N}$ such that
$\frac{N_2-N'_2}{N_2}>\frac{2-1}{2}=\frac{1}{2}$. And then there is
$x_2$ such that $\|x_2\|=R^{-M_1}\cdot2^{-2}\cdot\epsilon_1$ and
$$
 \lim\limits_{k\rightarrow\infty}\|T^kx_2\|=0, \ \ and \ \ \ \|
T^ix_2 \|\geq \gamma^i\|x_2\|, \ \ i=1,\ldots,N_2.
$$
So we can choose $M_2$ such that $\|T^nx_1\|+\|T^nx_2\|<\epsilon_2$
for any $n\geq M_2$.

Continue in this manner. If we have obtained $\{x_k\}_{k=1}^{m}$,
$\{N_k\}_{k=1}^{m}$, $\{N'_k\}_{k=1}^{m}$ and $\{M_k\}_{k=1}^{m}$
such that for each $k=2,\ldots,m$

(1) $\|x_k\|=R^{-M_{k-1}}\cdot2^{-k}\cdot\epsilon_{k-1}$ ;

(2) $\| T^ix_k \|\geq \gamma^i\|x_k\|, \ \ i=1,\ldots,N_k $;

(3) $\gamma^{N'_k}\cdot
R^{-M_{k-1}}\cdot2^{-k}\cdot\epsilon_{k-1}>1$;

(4) $\frac{N_k-N'_k}{N_k}>\frac{k-1}{k}$;

(5) $\sum\limits_{j=1}^k\|T^nx_j\|<\epsilon_k$, for any $n\geq M_k$;

Select $N'_{m+1}\in\mathbb{N}$ with $\gamma^{N'_{m+1}}\cdot
R^{-M_m}\cdot2^{-(m+1)}\cdot\epsilon_m>1$. Consequently, we have
$N_{m+1}\in\mathbb{N}$ such that
$\frac{N_{m+1}-N'_{m+1}}{N_{m+1}}>\frac{m+1-1}{m+1}=\frac{m}{m+1}$.
And then there is $x_{m+1}$ such that
$\|x_{m+1}\|=R^{-M_m}\cdot2^{-(m+1)}\cdot\epsilon_m$ and
$$
 \lim\limits_{k\rightarrow\infty}\|T^kx_{m+1}\|=0, \ \ and \ \ \ \|
T^ix_{m+1}\|\geq \gamma^i\|x_{m+1}\|, \ \ i=1,\ldots,N_{m+1}.
$$
So we can choose $M_{m+1}$ such that
$\sum\limits_{j=1}^{m+1}\|T^nx_j\|<\epsilon_{m+1}$ for any $n\geq
M_{m+1}$.

Therefore, we obtain a sequence of points $\{x_k\}_{k=1}^{\infty}$
associated with three sequences of integers
$\{N_k\}_{k=1}^{\infty}$, $\{N'_k\}_{k=1}^{\infty}$ and
$\{M_k\}_{k=1}^{\infty}$ satisfying conditions (I-V). Moreover,
conditions (I-III) imply following statement:

(VI) $\sum\limits_{k=1}^{\infty}\| x_k \|$ is finite.

(VII) For each $p$, $\| T^ix_k \|< 2^{-k}\epsilon_{k-1}$, for any
$k>p$ and any $1\leq i \leq M_p$. Hence,
$\sum\limits_{k=p+1}^{\infty}\| T^ix_k
\|<\sum\limits_{k=p+1}^{\infty} 2^{-k}\epsilon_{k-1}<\epsilon_p$,
for any $1\leq i \leq M_p$.

(VIII) For each $k$, $\| T^ix_k \|\geq 1$, $i=N'_k,\ldots,N_k$.
\par
\par
Notice $M_k>N_k>N'_k>M_{k-1}$ for each $k$ by the manner of our
construction. Then we have

(V') $\sum\limits_{j=1}^{k-1}\|T^nx_j\|<\epsilon_{k-1}$, for
$n=N'_{k},\ldots,N_{k}.$

(VII') For each $p$, $\sum\limits_{k=p+1}^{\infty}\| T^nx_k
\|<\epsilon_p$, $n=N'_{p},\ldots,N_{p}. $

Let $\Sigma_2=\{0,1\}^\mathbb{N}$ be a symbolic space with two
symbols. According to condition (VI), we can define a map $f:
\Sigma_2\rightarrow X$ as follows,
$$
f(\xi)=\sum\limits_{k=1}^{\infty}\xi_kx_k,
$$
for every element $\xi=(\xi_1,\xi_2, \ldots)\in \Sigma_2$.

Obviously one can get an uncountable subset $D\in \Sigma_2$ such
that for any two distinct $\xi, \xi'\in D$, $\xi$ and $\xi'$ have
infinite coordinates different and infinite coordinates equivalent.
Then
$$
d(f(\xi),
f(\xi'))=\|f(\xi)-f(\xi)\|=\|\sum\limits_{k=1}^{\infty}(\xi_k-\xi'_k)
x_k \|.
$$

Set $\theta=(\theta_1,\theta_2,
\ldots)=(\xi_1-\xi'_1,\xi_2-\xi'_2,\ldots)$. Then $ d(f(\xi),
f(\xi'))=\|\sum\limits_{k=1}^{\infty}\theta_k x_k \|$. Note that the
possible values of $\xi_k-\xi'_k$ are only 0, -1 or 1, and $\theta$
has infinite coordinates being zero and infinite coordinates being
nonzero.

Now we'll prove that $\{f(\xi),f(\xi')\}$ is a distributionally
chaotic pair.

Let $z=\sum\limits_{k=1}^{\infty}\theta_k x_k $. Suppose
$\{k_q\}_{q=1}^{\infty}$ be the infinite subsequence such that the
$k_q-th$ coordinate of $\theta$ is nonzero(1 or -1) and
$\{k_r\}_{r=1}^{\infty}$ be the infinite subsequence such that the
$k_r-th$ coordinate of $\theta$ is zero.

By (V'), (VII') and (VIII), for $n=N'_{k_q},\ldots,N_{k_q}$
$$
\| T^nz \|\geq \| T^n(\theta_{k_q}x_{k_q})
\|-\sum\limits_{j=1}^{k_q-1}\|T^nx_j\|-\sum\limits_{j=k_q+1}^{\infty}\|
T^nx_j \|>1-\epsilon_{k_q-1}-\epsilon_{k_q}.
$$
Since $\{\epsilon_k\}_{k=1}^{\infty}$ decrease to zero, then
\begin{eqnarray*}
F_{f(\xi)f(\xi')}(\frac{1}{2})&=&\liminf\limits_{n\rightarrow\infty}F^{n}_{f(\xi)f(\xi')}(\frac{1}{2})
\\ &\leq&
\liminf\limits_{q\rightarrow\infty}F^{N_{k_q}}_{f(\xi)f(\xi')}(\frac{1}{2})
\\ &\leq& \lim\limits_{q\rightarrow\infty}\frac{N'_{k_q}}{N_{k_q}} \\
&\leq&
\lim\limits_{q\rightarrow\infty}\frac{1}{k_q}=0.
\end{eqnarray*}
On the other hand, for $n=N'_{k_r},\ldots,N_{k_r}$
$$
\| T^nz \|\leq \| T^n(\theta_{k_r}x_{k_r})
\|+\sum\limits_{j=1}^{k_r-1}\|T^nx_j\|+\sum\limits_{j=k_r+1}^{\infty}\|
T^nx_j \|\leq\epsilon_{k_r-1}+\epsilon_{k_r}.
$$
Since $\{\epsilon_k\}_{k=1}^{\infty}$ decrease to zero, then for any
$\tau>0$
\begin{eqnarray*}
F_{f(\xi)f(\xi')}^{*}(\tau)&=&\limsup\limits_{n\rightarrow\infty}F^{n}_{f(\xi)f(\xi')}(\tau)
\\ &\geq&
\limsup\limits_{r\rightarrow\infty}F^{N_{k_r}}_{f(\xi)f(\xi')}(\tau)
\\ &\geq&
\lim\limits_{r\rightarrow\infty}\frac{N_{k_r}-N'_{k_r}+1}{N_{k_r}}
\\ &\geq& \lim\limits_{r\rightarrow\infty}\frac{k_r-1}{k_r}=1.
\end{eqnarray*}
Therefore, $\{f(\xi),f(\xi')\}$ is a distributionally chaotic pair
for any distinct points $\xi, \xi' \in D$ and hence $f(D)$ is a
distributionally  $\epsilon$-scrambled set. This ends the proof.

\end{proof}

In the present article, it has been filled by this conclusion. In
fact, we could extend it in some way. For instance, we have

\begin{theorem}[Weakly Distributionally Chaotic Criterion]\label{W-D-C-C}
Let $X$ be a Banach space and let $T\in \mathcal {L}(X)$. If for any
sequence of positive numbers $C_m$ increasing to $+\infty$, there
exist $\{x_m\}_{m=1}^{\infty}$ in $X$ satisfying

$(WNU1) \ \ \  \ \ \lim\limits_{k\rightarrow\infty}\|T^kx_m\|=0$.

$(WNU2)$  \ \ There is a sequence of positive integers $N_m$
increasing to $+\infty$, such that
$\lim\limits_{m\rightarrow\infty}\frac{\# \{0\leq i\leq N_m; \|
T^ix_m \|\geq C_m\|x_m\| \} }{N_m}=1$.

Then $T$ is distributionally chaotic.

\end{theorem}

The proof is similar to Theorem \ref{D-C-C}. According to
Grosse-Erdmann's characterization \cite{Grosse-Erdmann}, it is not
difficult to see that Devaney chaotic backward shift operators
satisfy the conditions of Theorem \ref{W-D-C-C}. So one can get the
following conclusion immediately.

\begin{corollary}\label{shift}
If $T$ is a Devaney chaotic backward shift operator, then $T$ is
distributionally chaotic.
\end{corollary}

\begin{remark}
This result is contained in a recent article \cite{Gim}, in which F.
Mart\'{\i}nez-Gim\'{e}nez, P. Oprocha and A. Peris, considered
distributional chaos for shift operators.
\end{remark}

Applying this Distributionally Chaotic Criterion \ref{D-C-C}, we'll
provide a sufficient condition which gives distributional chaots for
Cowen-Douglas operators.

\begin{theorem}\label{C-D-Operator}
Let $T\in \mathcal {B}_{n}(\Omega)$. If $\Omega\cap S\neq \phi$,
then $T$ is norm-unimodal. Consequently, $T$ is distributionally
chaotic.
\end{theorem}

\begin{proof}
Since $\Omega$ is a connected open subset and $\Omega\cap S\neq
\phi$, there exists $\beta\in\Omega$ with $|\beta|>1$. Furthermore,
we can select a nontrivial $y\in ker(T-\beta)$. Let
$1<\gamma<|\beta|$ be a constant. Given any $m\in\mathbb{N}$, set
$\epsilon<\frac{\|y\|}{2}\cdot\min\{1, \
\frac{|\beta|^i-\gamma^i}{|\beta|^i+1}, \ 1\leq i \leq m\}$. Then
$U=\bigcap\limits_{i=0}^{m}T^{-i}(B(T^{i}y,\epsilon))$ is an open
neighborhood of $y$. Then for any $z\in U$,
$$
\| T^iz \|\geq\| T^iy \|-\epsilon=|\beta|^i\|y\|-\epsilon
\geq|\beta|^i\|z\|-(|\beta|^i+1)\epsilon \geq \gamma^i\|z\|, \
i=1,\ldots,m
$$

By hypothesis, one can obtain a bounded connected open subset
$\Omega_{1}$ such that
$$\sup\limits_{\alpha\in\Omega_{1}}|\alpha|=\lambda<1.$$
By proposition \ref{local span}, there exists a point $x\in U$
satisfying
$$x=\sum\limits_{j=1}^{t}x_{j} ,$$
where $x_{j}\in  {ker(T-\lambda_{j})}, \ {\lambda_{j}\in
\Omega_{1}}, j=1,\ldots,t$. Then this $x$ is the point we hope to
get. One hand,
$$\lim\limits_{k\rightarrow\infty}\|T^kx\|=
\lim\limits_{k\rightarrow\infty}\|\sum\limits_{j=1}^{t}\lambda_{j}^{k}x_{j}\|\leq
\lim\limits_{k\rightarrow\infty}\sum\limits_{j=1}^{t}|\lambda_{j}^{k}|\cdot\|x_{j}\|
\leq
(\sum\limits_{j=1}^{t}\|x_{j}\|)\lim\limits_{k\rightarrow\infty}\lambda^{k}=0.$$
On the other hand, according to the previous statement and $x\in U$,
we have
$$
\| T^ix \|\geq \gamma^i\|x\|, \ \ i=1,2,\ldots,m.
$$
Therefore $T$ is norm-unimodal and hence $T$ is distributionally
chaotic by Theorem \ref{D-C-C}.
\end{proof}

\end{document}